\documentclass{elsarticle}

\usepackage{amssymb,amsthm,amsmath,hyperref}

\usepackage{pdfsync}

\usepackage{graphicx}

\usepackage{multicol}

\newtheorem{teo}{Theorem}
\newtheorem{lem}{Lemma}
\newdefinition{rmk}{Remark}
\newdefinition{rmks}{Remarks}
\newtheorem{cor}[teo]{Corollary}

\begin{document}

\title{On the zeros of orthogonal polynomials on the unit circle}

\author{Mar\'{\i}a Pilar Alfaro\fnref{thank}}
\ead{palfaro@unizar.es}
\address{Departamento de Matem\'aticas,
Universidad de Zara\-goza,\\
Calle Pedro Cerbuna s/n,
50009 Zaragoza, Spain}

\author{Manuel Bello-Hern\'andez\corref{cor1}\fnref{thank}}
\ead{mbello@unirioja.es}
\address{Dpto. de Matem\'aticas y Computaci\'on,
Universidad de La Rioja,\\
Edif. J. L. Vives, Calle Luis de Ulloa s/n,\\
26004 Logro\~no, Spain}

\author{Jes\'us Mar\'{\i}a Montaner\fnref{thank}}
\ead{montaner@unizar.es}
\address{Departamento de Matem\'atica Aplicada,
Universidad de Zara\-goza,\\
Edificio Torres Quevedo, Calle Mar\'{\i}a de Luna~3,
50018 Zaragoza, Spain}

\fntext[thank]{This research was supported in
part from `Ministerio de Ciencia y Tecnolog\'{\i}a', Project
MTM2009-14668-C02-02}

\cortext[cor1]{Corresponding author}


\begin{abstract}Let $\{z_n\}$ be a sequence in the unit disk $\mathbb{D}=\{z\in\mathbb{C}:|z|<1\}$. It is known that there exists a unique positive Borel measure in the unit circle $\Gamma=\{z\in\mathbb{C}:|z|=1\}$ such that the orthogonal polynomials $\{\Phi_n\}$ satisfy
\[
\Phi_n(z_n)=0
\]
for each $n=1,2,\ldots$. Characteristics of the orthogonality measure and asymptotic properties of the orthogonal polynomial are given in terms of asymptotic behavior of the sequence $\{z_n\}$. Particular attention is paid to periodic sequence of zeros $\{z_n\}$  of period two and three.
\end{abstract}


\maketitle


\section{Introduction}

A figure which displays the zeros of an orthogonal polynomial on the unit circle (OPUC) lets us state some properties of the Verblunsky coefficients and other parameters of OPUC (see Figures \ref{GrafCerosPerDos}--\ref{GrafCerosPerTres2} and Section 8.4 of \cite{SimonTI}). The zeros of OPUCs are eigenvalues of many operators. So, conclusions on the measure and other properties of OPUCs in terms of information about the zeros are interesting.

In the last decade several papers on zeros of OPUC have been published. For instance, we have \cite{MarMclSaf}, \cite{MarMclSaf2}, \cite{Simon1}, \cite{Simon2} and \cite{Simon3}. These articles joint to \cite{Bar-Lop-Saf}, \cite{MhaSaf}, \cite{NevTot}, \cite{Peh-Ste} and the seminal books of Simon, \cite{SimonTI} and \cite{SimonTII}, bring us closer to a better understanding of the properties of the zeros of OPUCs. However, there are several open questions about the zeros of OPUCs, see for example pp. 97--98 of \cite{Simon1}. In \cite{MarMclSaf} and \cite{MarMclSaf2} the properties of the zeros are studied in terms of analytic properties of the orthogonality measure, while in \cite{Simon1}, \cite{Simon2} and \cite{Simon3} the information about the zeros is given in terms of Verblunsky coefficients.  Others interesting problems are the description of properties of the zeros of OPUCs in terms of other parameters which also characterize OPUCs. We will deal with some of these questions in this paper.

We need to introduce some notations to state our results. Let $\mu$ be a nontrivial probability measure on $[0,2\pi)$ and let $\varphi_{n}(z)=\varphi_{n}(z,\mu)=\kappa_nz^n+\ldots, n=0,1,\ldots$ denote their orthonormal polynomials with positive leading coefficients, $\kappa_n>0$,
\[
\langle \varphi_n,\varphi_m\rangle=\frac1{2\pi}\int\varphi_n(e^{i\theta})\overline{\varphi_m(e^{i\theta})}\, d\mu(\theta)=\left\{ \begin{array}{lr}1,&\textnormal{ si }n=m,\\ 0 ,&\textnormal{ si }n\ne m.\end{array}\right.
\]
Let $\Phi_n(z)=\frac{\varphi_{n}(z)}{\kappa_n}$ be the monic OPUC. Then
\begin{equation}\label{recurrencia}
\Phi_{n+1}(z)=z\Phi_n(z)+\Phi_{n+1}(0)\Phi_n^*(z),\quad n\ne 0
\end{equation}
with $\Phi_0(z)=1$. All zeros of $\Phi_n$ lie in the unit disk $\mathbb{D}=\{z\in\mathbb{C}:|z|<1\}$ and, therefore, $\Phi_{n+1}(0)\in\mathbb{D}$. Moreover, Verblunsky's Theorem (see Theorem 1.7.11, p. 97, of \cite{SimonTI}) states that given a sequence $\{\alpha_n:n=1,2,\ldots\}$ in $\mathbb{D}$ there exists a unique probability measure $\mu$ on $[0,2\pi)$ such that $\Phi(0,\mu)=\alpha_n,\, n=1,2,\ldots$.

If $\{z_n\}$ is a sequence in $\mathbb{D}$ such that
\begin{equation}\label{CondicionZn}
\Phi_n(z_n)=0, \quad n=1,2,\ldots,
\end{equation}
then (\ref{recurrencia}) yields
\begin{equation}\label{ReflexionCoefZeros}
\Phi_{n+1}(0)=-z_{n+1}\frac{\Phi_n(z_{n+1})}{\Phi_{n}^*(z_{n+1})}
\end{equation}
and Verblunsky's Theorem tells us that there exists a unique orthogonality measure. So, the OPUC are uniquely determined by a sequence of their zeros, i.e., by a sequence $\{z_n\}$ such that (\ref{CondicionZn}) holds.

In this paper we obtain properties of OPUC in term of properties of a sequence of their zeros. For example, we prove the following result about OPUC with periodic zeros.

\begin{teo}\label{cerosPerTres} Suppose that there exists a common zero for $\Phi_{n}$ and $\Phi_{n-3}$  for all $n$ large enough. Let  $\zeta_j:j=1,2,3$, be such common zeros, i.e., there exist $n_0$ such that for all $n\ge n_0$,
\[
\Phi_n(\zeta_j)=0,\quad n=j\quad \textnormal{mod } 3.
\]
If $r\overset{\textnormal{def}}{=}\max\{|\zeta_j|:j=1,2,3\}\le\frac{-1+\sqrt{5}}{2}$,
\[
\lim_{n\to\infty}\Phi_n(0)=0.
\]
If $r<\frac{-1+\sqrt{5}}{2}$,
\[
\limsup_{n\to\infty}|\Phi_n(0)|^{1/n}\le \frac{r^2}{1-r}<1.
\]
\end{teo}
The numerical experiments show that when the three common zeros have magnitude greater than $\frac{-1+\sqrt{5}}{2}$, then the zeros are uniformly distributed on three arcs of unit
circle like those corresponding to a measure supported on three arcs, see Figures \ref{GrafCerosPerTres} and \ref{GrafCerosPerTres2}. If the sequence $\{z_n\}$ is periodic of period two, then the Verblunsky coefficients are an asymptotic periodic sequence, a situation studied in \cite{Bar-Lop-Saf} and \cite{Simon2}.

A completely different situation to  $\{z_n\}$ period takes place when it is dense in $\mathbb{D}$ and so also the zeros of the OPUC are dense in $\mathbb{D}$. This case was studied by Khrushchev in \cite{Kru}. He showed that there exist OPUCs with dense zeros in $\mathbb{D}$ with orthogonality measure in  many classes of measures, including Szeg\H{o} measures, measures with absolutely
convergent series of their Verblunsky parameters (see also Example 1.7.18, p. 98, of \cite{SimonTI}). We prove that the zeros can not be ``approach the unit circle too fast'' if and only if the orthogonality measure lies in the Nevai class.

\begin{teo}\label{TeoVelCirOPUC} The following statements are equivalent
\[
\lim_{n\to\infty}\Phi_n(0)=0
\]
and
\[
\lim_{n\to\infty}\sum_{j=1}^{n}(1-|z_{n,j}|)=\infty,
\]
where $\{z_{n,j}:j=1,\ldots,n\}$ are the zeros of $\Phi_n$.
\end{teo}

The proof of this theorem is included in Section \ref{SecDisZerUC}. This section also contains a study of the rate with which the zeros of OPUCs of the Chebyshev weight on an arc of the unit circle approach to $\partial(\mathbb{D})$. This example joints the cases studied in \cite{MarMclSaf2} of measure with a finite number of singularities in the unit circle and analytic measure corroborate  Theorem \ref{TeoVelCirOPUC}. In Section \ref{SecPropGen} we make some remarks about general properties of OPUCs in terms of the sequence of zeros stated. In Section \ref{SeccionDosCerosPerid} we study OPUC with periodic zeros of period two, it is $\{z_n\}$ has period two. The proof of Theorem  \ref{cerosPerTres} is contained in Section \ref{SeccionTresCerosPerid}. We include several figures which display zeros of OPUC for different sequences $\{z_n\}$ stated.

\section{General properties}\label{SecPropGen}

Let $\{z_n:n=1,2,\ldots\}$ be a sequence in $\mathbb{D}$ and let $\{\Phi_n\}$ be the sequence of monic orthogonal polynomials satisfying (\ref{CondicionZn}), i.e.,
\[
\Phi_{n+1}(0)=-z_{n+1}\frac{\Phi_n(z_{n+1})}{\Phi_{n}^*(z_{n+1})}=-z_{n+1} \frac{z_{n+1}-z_n}{1-\overline{z_n}z_{n+1}}\prod_{j:z_{n,j}\neq z_n}\frac{z_{n+1}-z_{n,j}}{1-\overline{z_{n,j}}z_{n+1}}.
\]
As $\left|\frac{z_{n+1}-z_{n,j}}{1-\overline{z_{n,j}}z_{n+1}}\right|<1$, the following result follows.

\begin{lem} If $\lim \frac{z_{n+1}-z_n}{1-\overline{z_n}z_{n+1}}=0,$ then $\lim_n\Phi_n(0)=0.$
In particular, if $\lim z_n=z_0, \quad|z_0|<1$,  $\lim_n\Phi_n(0)=0$. Moreover, if $w_n=|z_n-z_0|$,
\begin{equation}\label{cotaRaizN-esimaEnTerminosVelocidad}
\limsup|\Phi_n(0)|^{1/n}\le \limsup w_n^{1/n}.
\end{equation}
If $\displaystyle\nu=\lim_{n\to\infty,\, n\in\Lambda} \nu_{\Phi_n}$ \footnote{$\nu_{\Phi_n}\overset{\textnormal{def}}{=}\frac1n\sum_{j=1}^{n}\delta_{\{z_{n,j\}}}$. The limit of a sequence of measures throughout this paper is taken in the weak-* topology.}, $\lim z_n=z_0$, and $\nu(\{z_0\})=0$,
\begin{equation}\label{cotaRaizN-esima}
\limsup_{n\to\infty,\, n\in\Lambda}|\Phi_{n+1}(0)|^{1/(n+1)}\le \exp\int\log\left|\frac{z_0-\zeta}{1-\overline{\zeta}z_0}\right|d\nu.
\end{equation}
\end{lem}

\begin{proof}
We only check (\ref{cotaRaizN-esima}) because of the other statements are trivial. Let $\delta>0$ and $L=\limsup_{n\to\infty,\, n\in\Lambda}|\Phi_{n+1}(0)|^{1/(n+1)}$. Then
\[
\Phi_{n+1}(0)=-z_{n+1}\prod_{j:|z_{n,j}-z_0|< \delta}\frac{z_{n+1}-z_{n,j}}{1-\overline{z_{n,j}}z_{n+1}} \prod_{j:|z_{n,j}-z_0|\ge \delta}\frac{z_{n+1}-z_{n,j}}{1-\overline{z_{n,j}}z_{n+1}}
\]
\begin{multline*}
\Rightarrow\frac1{n+1}\log|\Phi_{n+1}(0)|\le \frac1{n+1}\sum_{j:|z_{n,j}-z_0|\ge \delta}\log\left|\frac{z_{n+1}-z_{n,j}}{1-\overline{z_{n,j}}z_{n+1}}\right|\\=\frac{n}{n+1} \int_{|\zeta-z_0|\ge\delta} \log\left|\frac{z_{n+1}-\zeta}{1-\overline{\zeta}z_{n+1}}\right|d\nu_{\Phi_n}(\zeta).
\end{multline*}
Since
\[
\lim_n\log\left|\frac{z_{n+1}-\zeta}{1-\overline{\zeta}z_{n+1}}\right|=
\log\left|\frac{z_{0}-\zeta}{1-\overline{\zeta}z_{0}}\right|
\]
uniformly on $\textnormal{sop}(\nu)\cap \{\zeta:|\zeta-z_0|\ge\delta\}$ and $\lim_{n\in\Lambda}\nu_{\Phi_n}=\nu$, we have
\[
\log L\le \int_{|\zeta-z_0|\ge\delta}\log\left|\frac{z_{0}-\zeta}{1-\overline{\zeta}z_{0}}\right|d\nu(\zeta),
\]
since $\delta>0$ is arbitrary and $\left|\frac{z_{0}-\zeta}{1-\overline{\zeta}z_{0}}\right|\le1$ for all
$\zeta:|\zeta|\le 1$, the integral above is non-positive and monotone decreasing as  function of $\delta$. Thus, the proof finishes using hypothesis $\nu(\{z_0\})=0$.
\end{proof}

\begin{rmks}
\begin{enumerate}
\item In \cite{MhaSaf} it proved that if $\lim_{n\to\infty}\sum_{j=1}^n\Phi_j(0)=0$ and $\Lambda$ is an infinite subset of natural numbers such that
     \[
    \lim_{n\to\infty,\, n\in\Lambda}|\Phi_n(0)|^{1/n}=\limsup|\Phi_n(0)|^{1/n}\overset{\textnormal{def}}{=}L,
    \]
    then
    \[
    \lim_{n\to\infty,\, n\in\Lambda}\nu_{\Phi_n}=m_L,
    \]
    where $m_L$ is the Lebesgue measure on the circle of radius $L$ (the Mhaskar-Saff circle). From (\ref{cotaRaizN-esimaEnTerminosVelocidad}), $\limsup|w_n|^{1/n}$ is a bound of the radius $L$ where the zeros of the polynomials of degree  $n\in\Lambda$  are uniformly distributed.

\item If $\limsup|\Phi_n(0)|^{1/n}<1$, in \cite{Simon1}, Simon proved that the rate of convergence of the zeros to the Nevai-Totik points is geometric. So this rate is slower than the radius of the Mhaskar-Saff circle.

\item If $L>0$ and $|z_0|\le L$, then
    \begin{equation}\label{EquationAux9}
    \int\log\left|\frac{z_0-\zeta}{1-\overline{\zeta}z_0}\right|dm_L(\zeta)=\log L.
    \end{equation}
    Thus, (\ref{cotaRaizN-esima}) yields
\[
\limsup_{n\to\infty,\, n\in\Lambda}|\Phi_{n+1}(0)|^{1/(n+1)}
\]
less than or equal to the infimum of $L$ such that $m_L$ is weak-* limit of some convergent subsequence $\{\nu_{\Phi_n}\}$ and
\[
\{z:|z|\le L\}\cap\left(\bigcup_{k=1}^{\infty}\overline{\bigcap_{n=k}^{\infty}\{z:\Phi_n(z)=0\}}\right) \ne\emptyset.
\]
\end{enumerate}
\end{rmks}

Since $\Phi_{n+1}(0)=(-1)^{n+1}\prod_j z_{n+1,j}$, (\ref{ReflexionCoefZeros}) implies
\[
\Phi_{n+1}(0)^{n}=\prod_j\frac{\Phi_n(z_{j,n+1})}{\Phi_n^*(j,z_{n+1})}
\]
\[
\Leftrightarrow |\Phi_{n+1}(0)|^{1/(n+1)}=\exp\iint\log\left|\frac{z-w}{1-\overline{w}z}\right| d\nu_{\Phi_{n+1}}(z)d\nu_{\Phi_n}(w)
\]
So, in addition to (\ref{cotaRaizN-esima}) we have the following

\begin{lem} If $\nu_1=\lim_{n\to\infty,n\in\Lambda}\nu_{\Phi_n}$ and $\nu_2=\lim_{n\to\infty,n\in\Lambda}\nu_{\Phi_{n+1}}$, then
\begin{equation}\label{cotaRaizN-esima2}
\limsup_{n\to\infty,n\in\Lambda}|\Phi_{n+1}(0)|^{1/(n+1)}\le \iint\log\left|\frac{z-w}{1-\overline{w}z}\right| d\nu_2(z)d\nu_1(w).
\end{equation}
\end{lem}

\begin{proof} The function $f(z,w)=\log\left|\frac{z-w}{1-\overline{w}z}\right| $ is non-positive upper semicontinuous in $\overline{\mathbb{D}}\times\overline{\mathbb{D}}$, so there is a monotone decreasing sequence of non-positive continuous function $\{g_m\}$ such that $f(z,w)=\lim_mg_m(z,w)$ pointwise in $\overline{\mathbb{D}}\times\overline{\mathbb{D}}$ (see Theorem 1.1, p.1, in \cite{Saf-Tot}). Thus,
\begin{multline*}
|\Phi_{n+1}(0)|^{1/(n+1)}= \exp\iint\log\left|\frac{z-w}{1-\overline{w}z}\right| d\nu_{\Phi_{n+1}}(z)d\nu_{\Phi_n}(w)\\ \le \exp\iint g_m(z,w) d\nu_{\Phi_{n+1}}(z)d\nu_{\Phi_n}(w),
\end{multline*}
and since $\lim_{n\to\infty}(\nu_{\Phi_{n+1}}\times\nu_{\Phi_{n}})=\nu_2\times\nu_1$, by the monotone convergence theorem the conclusion follows immediately.
\end{proof}

\begin{rmk}
According to (\ref{EquationAux9}), if $\limsup_{n\to\infty}|\Phi_{n+1}(0)|^{1/(n+1)}=L$ and $\nu_1=\nu_2=m_L$, then (\ref{cotaRaizN-esima2}) becomes an equality.
\end{rmk}

\section{Zeros of period two}\label{SeccionDosCerosPerid}
\begin{lem}\label{LemVerCerPerDos} If the sequence $\{z_n\}$ is periodic of period two, i.e.,
\begin{equation}\label{cerPerDos}
z_n=\left\{\begin{array}{lr}\alpha_1,&  $n$ \textnormal{ odd} \\ \alpha_2,& n \textnormal{ even} \end{array}\right.
\end{equation}
with $\{\alpha_1,\alpha_2\}\subset\mathbb{D}$, then
\begin{equation}\label{coefRecDos1}
\Phi_{1}(0)=-\alpha_1,\quad \Phi_2(0)=-\alpha_2\, C_{\alpha_1,\alpha_2},
\end{equation}
and for all $n\ge 3$,
\begin{equation}\label{coefRecDos2}
\Phi_{n}(0)=(-1)^{n-1}C_{\alpha_1,\alpha_2} \left\{\begin{array}{lr}\alpha_1^{(n-1)/2}\alpha_2^{(n-1)/2},& n \textnormal{ odd},\\ \alpha_1^{-1+n/2}\alpha_2^{n/2},& n \textnormal{ even},\end{array}\right.
\end{equation}
where $C_{\alpha_1,\alpha_2}=\frac{\alpha_2-\alpha_1}{1-\overline{\alpha_1}\alpha_2}$.
\end{lem}

\begin{proof} Iterating (\ref{recurrencia}), we obtain
\begin{equation}\label{recurrencia2}
\Phi_{n+1}(z)= z\left(z+\overline{\Phi_n(0)}\Phi_{n+1}(0)\right)\Phi_{n-1}(z)+ \left(\Phi_{n+1}(0)+z\Phi_n(0)\right)\Phi_{n-1}^*(z),
\end{equation}
$n\ge 2$. If $\Phi_{n+1}$ and $\Phi_{n-1}$ have a common zero, $\zeta$, then setting $z=\zeta$ we get
\[
\Phi_{n+1}(0)=-\zeta\Phi_n(0)
\]
which proves the lemma.
\end{proof}

\begin{rmk} If $\min\{|\alpha_1|,|\alpha_2|\}=0$ or $\alpha_1=\alpha_2$, $\Phi_n(0)=0$ for all $n\ge 3$ and
\[
\Phi_n(z)=z^{n-2}\Phi_2(z),\quad\forall n\ge 3.
\]
Thus, we will assume throughout this section that
\[
\min\{|\alpha_1|,|\alpha_2|\}>0\quad \textnormal{and}\quad \alpha_1\ne \alpha_2.
\]
\end{rmk}

In \cite{Bar-Lop-Saf} it is studied OPUCs with Verblunsky coefficients satisfying $\lim_n\Phi_n(0)=0$ and there exists a natural number $k$ such that
\begin{equation}\label{PeriodVerCoef}
\lim_{n\to\infty,\, n=j\, \textnormal{mod }k}\frac{\Phi_{n+1}(0)}{\Phi_{n}(0)}\, \textnormal{exists, }\quad j=1,2,\ldots,k.
\end{equation}
From Lemma \ref{LemVerCerPerDos}, (\ref{PeriodVerCoef}) holds when $\{z_n\}$ is periodic of period two.

In \cite{Simon2} going on OPUCs satisfying (\ref{PeriodVerCoef}). There, it is required that there exists $\Delta\in(0,1)$ such that
\begin{equation}\label{SimonCoefRef}
\Phi_n(0)=\sum_{j=1}^lC_jb_j^{n}+O(\Delta b^n)
\end{equation}
where $0\not\in\{C_j\}$, $\{b_j\}$ are distinct and $|b_j|=|b|<1,\, j=1,\ldots,l$. In our case, if the sequence $\{z_n\}$ is periodic of period two, then (\ref{SimonCoefRef}) holds with
\[
C_1=-\frac{C_{\alpha_1,\alpha_2}}{2}(\alpha_1+\frac1{\sqrt{\alpha_1\alpha_2}}),\quad C_2=-\frac{C_{\alpha_1,\alpha_2}}{2}(\alpha_1-\frac1{\sqrt{\alpha_1\alpha_2}})
\]
and
\[
b_1=\sqrt{\alpha_1\alpha_2},\quad b_2=-\sqrt{\alpha_1\alpha_2}.
\]
Therefore, all results proved in \cite{Simon2} also hold for OPUCs with two periodic zeros. For example,

\begin{cor}\label{CorAsinDentro} If $\{z_n\}$ satisfies (\ref{cerPerDos}),
\[
\lim_{k}\frac{\Phi_{2k}(z)}{\alpha_1^k\alpha_2^kC_{\alpha_1,\alpha_2}} =\frac{D(0)D(z)^{-1}}{(\alpha_1\alpha_2-z^2)}(z-\alpha_2),
\]
\[
\lim_{k}\frac{\Phi_{2k+1}(z)}{\alpha_1^k\alpha_2^kC_{\alpha_1,\alpha_2}} =\frac{\alpha_2 D(0)D(z)^{-1}}{(\alpha_1\alpha_2-z^2)}(\alpha_1-z),
\]
uniformly on each compact subset of $\{z:|z|<\sqrt{|\alpha_1\alpha_2|}\}$, where
\begin{equation}\label{funcSzego}
D(z)=D(z,\mu)=\exp{\left(\frac1{2\pi}\int\frac{e^{i\theta}+z} {e^{i\theta}-z}\log(\mu'(e^{i\theta}))d\theta\right)}
\end{equation}
is the Szeg\H{o} function. As
\[
\kappa\overset{\textnormal{def}}{=}\lim_n\kappa_n=D(0)^{-1}=\prod_{j=1}^{\infty}(1-|\Phi_{j}(0)|^2)^{-1/2}<\infty,
\]
the analogous results also hold for orthonormal polynomials.
\end{cor}

In particular, if (\ref{cerPerDos}) is satisfied
\[
\lim_{n\to\infty}\frac{\Phi_{n+2}(z)}{\Phi_{n}(z)}=\alpha_1\alpha_2
\]
uniformly on compact subset of $\{z:|z|<\sqrt{|\alpha_1\alpha_2|}\}\setminus\{\alpha_1,\alpha_2\}$. Actually, we have only to delete from the disk $\{z:|z|<\sqrt{|\alpha_1\alpha_2|}\}$ that value $\alpha_1$ or $\alpha_2$ of lower magnitude. This result is proved in \cite{Bar-Lop-Saf} under the more general condition (\ref{PeriodVerCoef}).

It is worthwhile asymptotic behavior in an annulus about the critical circle $\{z:|z|=|\sqrt{|\alpha_1\alpha_2|}\}$ . It requires a parameter  $\Delta_1$ associated to a fine look at of the Verblunsky coefficients. Doing again the calculations in \cite{Simon2}, we obtain $\Delta_1=\sqrt{|\alpha_1\alpha_2|}$ and the following result.

\begin{teo}\label{TeoRegionCritica} If (\ref{cerPerDos}) holds, $D^{-1}(z),\, |z|<1$, admites a meromorphic extension, $D_{int}^{-1}$, to $\{z:|z|<\frac1{|\alpha_1\alpha_2|}\}$ with exactly two poles at $\pm\frac1{\sqrt{\alpha_1\alpha_2}}$ which is analytic in $\{z:|z|<\frac1{\sqrt{|\alpha_1\alpha_2|}}\}$. Moreover,
\[
\lim_n\Phi_n^*(z)=D(0)D_{int}^{-1}(z)
\]
uniformly on compact sets of $\{z:|z|<\frac1{\sqrt{|\alpha_1\alpha_2|}}\}$. Hence, in $\{z:|z|>\sqrt{|\alpha_1\alpha_2|}\}$,
\[
\lim_n\frac{\Phi_n(z)}{z^n}=D(0)\overline{D_{int}(1/\overline{z})}^{-1}.
\]

Moreover, let
\[
R_n(z)=\Phi_{n+1}(0)\left(\Phi_n^*(z)-D(0)D^{-1}(z)\right),
\]
and
\[
s_n(z)=\sum_{j=0}^{\infty}z^{-j-1}R_{n+j}(z).
\]
Then for all $\epsilon>0$,
\[
\max_{|z|\le 1}|R_n(z)|\le C ((|\alpha_1\alpha_2|+\epsilon)^n),
\]
the sums defining each $s_n$ converge in
\[
\mathbb{A}=\{z:|\alpha_1\alpha_2|<|z|<1\},
\]
in this set it defines an analytic function and if $\epsilon>0$ is sufficiently small,
\[
|s_n(z)|\le C \frac{(|\alpha_1\alpha_2|+\epsilon)^n}{|z|-|\alpha_1\alpha_2|-\epsilon}.
\]
For $z\in \mathbb{A}$, we have
\[
\Phi_{2k}(z)=s_{2k}(z)+\frac{C_{\alpha_1,\alpha_2}D(0)D(z)^{-1}\alpha_1^k\alpha_2^k} {(\alpha_1\alpha_2-z^2)}(z-\alpha_1)+z^{2k}D(0)\overline{D_{int}(1/\overline{z})}^{-1},
\]
\[
\Phi_{2k+1}(z) =s_{2k+1}(z)+\frac{\beta C_{\alpha_1,\alpha_2}D(0)D(z)^{-1}\alpha_1^k\alpha_2^k}{(\alpha_1\alpha_2-z^2)}(\alpha_1-z) +z^{2k+1}D(0)\overline{D_{int}(1/\overline{z})}^{-1}.
\]
\end{teo}

\begin{rmk}
Using the above result, in \cite{Simon2}, Simon also proved that the zeros of OPUCs satisfied what he called ``clock behavior'' in the Mhaskar-Saff circle $\{z:|z|=\sqrt{|\alpha_1\alpha_2|}\}$: the zeros approach to this circle with rate $O(\frac{\log n}{n})$, the rate of magnitudes of consecutive zeros is $1+O(\frac1{n\log n})$ and their are equally spaced with only a larger gap around $\pm\sqrt{\alpha_1\alpha_2}$. See Figures \ref{GrafCerosPerDos} and \ref{GrafCerosPerDos2}. These properties of the zeros let us to speculate looking for a justification of what it is seen in the drawings when $|\alpha_|\ne|\alpha_2|$, see Figure \ref{GrafCerosPerDos2}. For example, if $|\alpha_2|>|\alpha_1|$, then $\alpha_2$ is a Nevai-Totik point, there exists a zero near this point and another one, ``{\it which is not in the gaps}'', accumulates in other point in $\{z:\sqrt{|\alpha_1\alpha_2|}<|z|<1 \}$. Therefore, the Szeg\H{o} function has two zeros outside this disk which give some equilibrium with its singularities at $\pm\sqrt{\alpha_1\alpha_2}$ (remember $D(\infty)=D(0)^{-1}\ne 0$).
\end{rmk}

\subsection{On the meromorphic extension of the Szeg\H{o} function}

In Theorem \ref{TeoRegionCritica} we have required the meromorphic extension of the interior Szeg\H{o} function to $\mathbb{A}$. This function has two poles at $\pm\sqrt{\alpha\beta}$. Using technique of Fourier-Pad\'e approximants we prove that an extension with exactly two poles only can be do it to $\mathbb{A}$. To obtain that result we use a Lemma stated in \cite{Bar-Lop-Saf}.

Let $f\in L^1(\mu)$. Its Fourier expansion with respect to the orthonormal system
$\{\varphi_n\}$ is given by
\[
f(z) \sim \sum_{j=0}^{\infty} A_{j}\varphi_{j}(z),
\]
where $ A_j$ denotes the Fourier coefficient
\[
A_j=\langle f,\varphi_i\rangle.
\]
The Fourier-Pad\'e approximant of type $(n,m)$, $n,m\in\{0,1,\ldots\}$, of $f$ is the ratio $\pi_{n,m}(f)=p_{n,m}/q_{n,m}$ of any two polynomials $p_{n,m}$ and $q_{n,m}$ such that
\begin{enumerate}
\item[(i)] $\textnormal{deg}(p_{n,m})\le n$; $\textnormal{deg}(q_{n,m})\le m$, $q_{n,m}\not\equiv 0$.
\item[(ii)] $q_{n,m}(z)f(z)- p_{n,m})(z)\sim  A_{n,1}\varphi_{n+m+1}(z)+ A_{n,2}\varphi_{n+m+2}(z) +\ldots$.
\end{enumerate}
Condition (ii) above means that
\[
\langle q_{n,m}f - p_{n,m},\varphi_j\rangle=0
\]
for $j=0,\ldots,n+m$. In the sequel, we take $q_{n,m}$ with leading coefficient equal to 1.

The existence of such polynomials reduces to solving a homogeneous linear system of $m$ equations
on the $m + 1$ coefficients of $q_{n,m}$. Thus a nontrivial solution is guaranteed. In general, the rational function $\pi_{n,m}$ is not uniquely determined, but if for every solution of (i), (ii), the polynomial $q_{n,m}$ is of degree $m$, then $\pi_{n,m}$ is unique.

For $m$ fixed, a sequence of type $\{\pi_{n,m}, n\in \mathbb{N}\}$, is called an $m$th row of the Fourier-Pad\'e approximants relative to $f$. If $f$ is such that $R_0(f)>1$ and has in $\Delta_m(f)$ exactly $m$ poles then for all sufficiently large $n\ge n_0$, $\pi_{n,m}$ is uniquely determined and so is the sequence $\{\pi_{n,m}, \, n\ge n_0\}$. Here $\Delta_m(f) = \{z: |z|<R_m(f)\}$ is the largest disk centered at $z = 0$ in which $f$ can be extended to a meromorphic function with at most $m$ poles.
This and other results for row sequences of Fourier-Pad\'e approximants may be found in \cite{Sue}, \cite{Sue2} for Fourier expansion with respect to measures supported on an interval of the real line whose
absolutely continuous part with respect to Lebesgue's measure is positive almost everywhere.

The following result is stated in \cite{Bar-Lop-Saf}.

\vspace{0.5cm}

\begin{lem}\label{LemBar-Lop-Saf} Let $\mu$ be such that $R_0= R_0(D^{-1})>1$. The following assertions are equivalent:
\begin{enumerate}
\item[(a)] $D^{-1}$ has exactly $m$ poles in $\Delta_m= \Delta_m(D^{-1})$.
\item[(b)] The sequence $\{\pi_{n,m}(D^{-1});\,n=0,1,\ldots\}$ for all sufficiently large $n$ has exactly $m$ finite poles and there exists a polynomial $w_m(z)=z^m+\ldots$ such that
\[
\limsup_n\|q_{n,m}- w_m\|^{1/n}=\delta<1,
\]
where  $\|\cdot\|$ denotes any norm in the space of polynomials of degree at most $m$.
\end{enumerate}
The poles of $D^{-1}$ coincide with the zeros $z_1,\ldots,z_m$ of $w_m$, and
\begin{equation}\label{FormulaRm}
R_m=\frac1{\delta}\max_{1\le j\le m}|z_j|.
\end{equation}
\end{lem}

The lemma above lets us to prove the following result.

\begin{teo}\label{TeoFormulaRmD-1} If the sequence of zeros $\{z_n\}$ satisfies (\ref{cerPerDos}), then
\[
R_2(D^{-1})=\frac1{|\alpha\beta|}.
\]
\end{teo}

To prove the theorem above we need some calculations. Using recurrence relation (\ref{recurrencia}) it is easy to prove the following lemma.

\begin{lem}
\[
\langle z\varphi_{j},1\rangle=-\frac{\varphi_{j+1}(0)}{\kappa_j\kappa_{j+1}}.
\]

\[
\langle z\varphi_{j},\varphi_{m}\rangle=\left\{\begin{array}{lr}0,&\textnormal{ si } j<m-1,\\ \frac{\kappa_j}{\kappa_{j+1}},&\textnormal{ si }j=m-1,\\ -\frac{\varphi_{j+1}(0)\overline{\varphi_{m}(0)}}{\kappa_j\kappa_{j+1}},&\textnormal{ si }j> m-1.\end{array}\right.
\]
For $j<m-2$,
\[
\langle z^2\varphi_{j},\varphi_{m}\rangle=0,
\]
and
\[
\langle z^2\varphi_{m-2},\varphi_{m}\rangle=\frac{\kappa_{m-2}}{\kappa_m}.
\]
Moreover,
\[
\langle z^2\varphi_{m-1},\varphi_{m}\rangle= -\frac{\overline{\varphi_{m}(0)}}{\kappa_m^2}\left(\frac{\kappa_{m-1}}{\kappa_{m+1}}
\varphi_{m+1}(0)+\frac{\varphi_{m}(0)\overline{\varphi_{m-1}(0)}}{\overline{\varphi_{m}(0)}} \right).
\]
If $j\ge m$,
\begin{multline*}
\langle z^2\varphi_{j},\varphi_{m}\rangle= -\frac{\Phi_{j+1}(0)\overline{\varphi_{m}(0)}}{\kappa_j}\times\\ \times \left(  \frac{\overline{\Phi_{m-1}(0)}}{\overline{\Phi_{m}(0)}}+\frac{\Phi_{j+2}(0)}{\Phi_{j+1}(0)} -\sum_{l=m-1}^{j+1} \overline{\Phi_l(0)}\Phi_{l+1}(0)\right).
\end{multline*}
\end{lem}

It is known for measures on the Szeg\H{o} class ($\log\mu'\in L^1$) we have
\[
D^{-1}(z)=\frac1{\kappa}\sum_{j=0}^{\infty}\overline{\varphi_j(0)}  \varphi_j(z),\quad z\in\mathbb{D},
\]
see \cite{Ger}, p. 19; \cite{NevTot} Theorem 1; \cite{MhaSaf} Theorem 2.2; \cite{Bar-Lop-Saf}, p. 174. If (\ref{cerPerDos}) holds, the expansion above converges uniformly on compact subsets of $\{z:|z|<\frac1{\sqrt{|\alpha_1\alpha_2|}}\}$ and
\begin{equation}\label{DesFuncSze}
D_{int}^{-1}(z)=\frac1{\kappa}\sum_{j=0}^{\infty}\overline{\varphi_j(0)}  \varphi_j(z),\quad z\in\{z:|z|<\frac1{\sqrt{|\alpha_1\alpha_2|}}\}.
\end{equation}
Then, using the Lemma above and (\ref{DesFuncSze}), we obtain:

\begin{lem}\label{LemaFormulasCoefFourierFuncionSzegoZfuncSzego}
\[
\langle D_{int}^{-1},\varphi_{m}\rangle=\frac{\overline{\varphi_{m}(0)}}{\kappa},
\]
\[
\langle z D_{int}^{-1},\varphi_{m}\rangle=\frac{\overline{\varphi_{m}(0)}}{\kappa }\left( \frac{\overline{\Phi_{m-1}(0)}}{\overline{\Phi_{m}(0)}} - \sum_{j=m-1}^{\infty}\overline{\Phi_j(0)}\Phi_{j+1}(0)\right),
\]
\[
\langle z^2 D_{int}^{-1},\varphi_{m}\rangle=  \frac{\overline{\varphi_{m}(0)}}{\kappa}\left( \frac{\overline{\Phi_{m-2}(0)}}{\overline{\Phi_{m}(0)}}+O( \varphi_{m-1}(0))\right).
\]

\end{lem}

\begin{proof}{ of Theorem \ref{TeoFormulaRmD-1}.} Let us check first that there is not a monic polynomials of degree $1$ such that $q_{n,2}(z)=z-\tau_n$; i.e.,
\[
\langle q_{n,2}D_{int}^{-1},\varphi_{n+1}\rangle=\langle q_{n,2}D^{-1},\varphi_{n+2}\rangle=0.
\]
As $\langle q_{n,2}D_{int}^{-1},\varphi_{n+1}\rangle=0$, we have
\begin{equation}\label{EquationAux1}
\tau_n=\frac{\overline{\Phi_n(0)}}{\overline{\Phi_{n+1}(0)}}-\sum_{j=n}^{\infty}\overline{\Phi_j(0)}\Phi_{j+1}(0)
\end{equation}
Since also $\langle q_{n,2}D_{int}^{-1},\varphi_{n+2}\rangle=0$, we obtain
\begin{equation}\label{EquationAux2}
\langle zD_{int}^{-1},\varphi_{n+2}\rangle=\tau_n\langle D_{int}^{-1},\varphi_{n+2}\rangle
\end{equation}
From Lemma \ref{LemaFormulasCoefFourierFuncionSzegoZfuncSzego}, (\ref{EquationAux1}) and (\ref{EquationAux2}), we get
\[
\frac{\overline{\Phi_{n+1}(0)}}{\overline{\Phi_{n+2}(0)}}= \frac{\overline{\Phi_{n}(0)}}{\overline{\Phi_{n+1}(0)}} - \overline{\Phi_n(0)}\Phi_{n+1}(0)
\]
Since $\{\frac{\overline{\Phi_{n+1}(0)}}{\overline{\Phi_{n+2}(0)}}, \frac{\overline{\Phi_{n}(0)}}{\overline{\Phi_{n+1}(0)}}\}=\{-1/\overline{\alpha_1},-1/\overline{\alpha_2}\}$, $\alpha_1\neq\alpha_2$, and $\lim_{n}\overline{\Phi_n(0)}\Phi_{n+1}(0)=0$, the above relation is imposible for $n$ large enough.

Thus, the denominators, $q_{n,2}$, of Fourier-Pad\'e approximants of order $(n,2)$ are exactly of degree $2$ for $n$ large enough.

Let $q_{n,2}(z)=(z-\beta_n)(z-\tau_n)=z^2-(\beta_n+\tau_n)z+\beta_n\tau_n$. It satisfies
\[
\langle q_{n,2}D_{int}^{-1},\varphi_{n+1}\rangle=\langle q_{n,2}D_{int}^{-1},\varphi_{n+2}\rangle=0.
\]
Thus,
\[
(\beta_n+\tau_n)\langle z D_{int}^{-1},\varphi_{n+1}\rangle-\beta_n\tau_n \langle D_{int}^{-1},\varphi_{n+1}\rangle=\langle z^2D_{int}^{-1},\varphi_{n+1}\rangle,
\]
\[
(\beta_n+\tau_n)\langle z D_{int}^{-1},\varphi_{n+2}\rangle-\beta_n\tau_n \langle D_{int}^{-1},\varphi_{n+2}\rangle=\langle z^2D_{int}^{-1},\varphi_{n+2}\rangle.
\]
Hence,
\[
\beta_n+\tau_n=\frac{\left|\begin{array}{cc}\langle z^2 D_{int}^{-1},\varphi_{n+1}\rangle& -\langle D_{int}^{-1},\varphi_{n+1}\rangle\\ \langle z^2 D_{int}^{-1},\varphi_{n+2}\rangle& -\langle D_{int}^{-1},\varphi_{n+2}\rangle  \end{array}\right|}{\left|\begin{array}{cc}\langle z D_{int}^{-1},\varphi_{n+1}\rangle& -\langle D_{int}^{-1},\varphi_{n+1}\rangle\\ \langle z D_{int}^{-1},\varphi_{n+2}\rangle& -\langle D_{int}^{-1},\varphi_{n+2}\rangle  \end{array}\right|},
\]
\[
\beta_n\tau_n=\frac{\left|\begin{array}{cc}\langle z D_{int}^{-1},\varphi_{n+1}\rangle& \langle z^2 D_{int}^{-1},\varphi_{n+1}\rangle\\ \langle z D_{int}^{-1},\varphi_{n+2}\rangle& \langle z^2 D_{int}^{-1},\varphi_{n+2}\rangle \end{array}\right|}{\left|\begin{array}{cc}\langle z D_{int}^{-1},\varphi_{n+1}\rangle& -\langle D_{int}^{-1},\varphi_{n+1}\rangle\\ \langle z D_{int}^{-1},\varphi_{n+2}\rangle& -\langle D_{int}^{-1},\varphi_{n+2}\rangle  \end{array}\right|}.
\]
We have
\begin{multline}\label{EquationAux3}
\left|\begin{array}{cc}\langle z D_{int}^{-1},\varphi_{n+1}\rangle& -\langle D_{int}^{-1},\varphi_{n+1}\rangle\\ \langle z D_{int}^{-1},\varphi_{n+2}\rangle& -\langle D_{int}^{-1},\varphi_{n+2}\rangle  \end{array}\right|\\=
\frac{\overline{\varphi_{n+1}(0)\varphi_{n+2}(0)}}{\kappa^2 }\left(\frac{\overline{\Phi_{n+1}(0)}}{\overline{\Phi_{n+2}(0)}} -\frac{\overline{\Phi_{n}(0)}}{\overline{\Phi_{n+1}(0)}}+ \overline{\Phi_n(0)}\Phi_{n+1}(0) \right).
\end{multline}
Using, lemma above, we obtain there exists  $C\ne 0$ such that
\begin{equation}\label{EquationAux4}
\left|\begin{array}{cc}\langle z^2 D_{int}^{-1},\varphi_{n+1}\rangle& -\langle D_{int}^{-1},\varphi_{n+1}\rangle\\ \langle z^2 D_{int}^{-1},\varphi_{n+2}\rangle& -\langle D_{int}^{-1},\varphi_{n+2}\rangle  \end{array}\right|=C\frac{\overline{\varphi_{n+1}(0)\varphi_{n+2}(0)}}{\kappa^2}  \varphi_{n}(0).
\end{equation}
Combining (\ref{EquationAux3}) and (\ref{EquationAux4}), we have there is a constant $C'\ne 0$ such that
\[
\beta_n+\tau_n= C'\varphi_{n}(0)
\]

Doing the same calculation for $\beta_n\tau_n$, we obtain
\[
\beta_n\tau_n=\frac{\frac{\overline{\varphi_{n+1}(0)\varphi_{n+2}(0)}}{\kappa^2 } \left( \frac{\overline{\Phi_{n}(0)}}{\overline{\Phi_{n+1}(0)}} \frac{\overline{\Phi_{n}(0)}}{\overline{\Phi_{n+2}(0)}}- \frac{\overline{\Phi_{n-1}(0)}}{\overline{\Phi_{n}(0)}} \frac{\overline{\Phi_{n-1}(0)}}{\overline{\Phi_{n+1}(0)}} +O(\varphi_n(0))\right)}{\frac{\overline{\varphi_{n+1}(0)\varphi_{n+2}(0)}}{\kappa^2 }\left(\frac{\overline{\Phi_{n+1}(0)}}{\overline{\Phi_{n+2}(0)}} -\frac{\overline{\Phi_{n}(0)}}{\overline{\Phi_{n+1}(0)}}+ \overline{\Phi_n(0)}\Phi_{n+1}(0) \right)}
\]
\[
\lim_n \beta_n\tau_n=-\frac1{\overline{\alpha_1\alpha_2}}
\]
with geometric convergence with rate $\sqrt{|\alpha_1\alpha_2|}$.

Therefore,
\[
\|q_{n,2}(z)-(z^2-\frac1{\overline{\alpha_1\alpha_2}})\|^{1/n}=|\alpha_1\alpha_2|^{1/2},
\]
where the norm is anything in the space of polynomials of degree 2. From Lemma \ref{LemBar-Lop-Saf},
\[
R_2(D_{int}^{-1})=\frac{|z_j|}{|\alpha_1\alpha_2|^{1/2}}=\frac1{|\alpha_1\alpha_2|}
\]
where $z_j$ are the roots of $z^2-\frac1{\overline{\alpha_1\alpha_2}}$ (both have magnitude $\frac1{|\alpha_1\alpha_2|^{1/2}}$). It is $\{z:|z|<\frac1{|\alpha_1\alpha_2|}\}$ is the largest disk centered at $z = 0$ in which $D_{int}^{-1}(z)$ can be extended to a meromorphic function with at most two poles.

\end{proof}

\begin{rmk} An alternative proof of Theorem \ref{TeoFormulaRmD-1} is a Hadamard formula for $R_m(D^{-1})$ given in \cite{Bar-Lop-Saf2}. One of such formula is written in term of $\varphi_{n+k}^{(j)}(0).$ These can be obtained using Corollary \ref{CorAsinDentro}, then the unknown expressions
\[
c_{m}=\int e^{-im\theta}\log w(e^{i\theta})d\theta,
\]
appear. The values $R_m(D^{-1}),\, j=0,1,2$ founded
\[
R_0(D^{-1})=R_1(D^{-1})=\frac1{\sqrt{|\alpha_1\alpha_2|}},\quad R_2(D^{-1})=\frac1{|\alpha_1\alpha_2|},
\]
let us to obtain $c_0$ and $c_1$.
\end{rmk}

\begin{rmk} From the proof we obtain also that the Fourier-Pad\'e approximants of type $(n,1)$ of $D^{-1}$  has exactly a pole at $\overline{\alpha_1}^{-1}$ or $\overline{\alpha_2}^{-1}$ according $n$ is even or odd. Thus, they converge to $D_{int}^{-1}$ in $\{z:|z|<\frac1{\sqrt{|\alpha_1\alpha_2|}}\}$.
\end{rmk}

\section{Zeros of period three}\label{SeccionTresCerosPerid}
If the sequence of zeros, $\{z_n\}$, is periodic of period three, the Verblunsky coefficients are not a geometric progression as one might naively expect. This case is more complex: if the three periodic zeros have magnitudes at most $\frac{-1+\sqrt{5}}{2}$, then the measure is in the Nevai class, i.e., $\lim_{n\to\infty}\Phi_n(0)=0$, while if the periodic zeros have magnitudes greater than $\frac{-1+\sqrt{5}}{2}$, some numerical experiments show that, for large degree, the zeros of OPUCs  are close to three arcs of the unit circle, so the orthogonality measure should be supported on these arcs. See Figures \ref{GrafCerosPerTres} and \ref{GrafCerosPerTres2}.

To prove Theorem \ref{cerosPerTres}, we need the following lemma whose proof is easy, so we omit it.

\begin{lem}
Let $\{a_k:k\ge 2\}$ be the sequence
\[
a_2=a_3=\frac{2r^2}{1+r^2},\quad a_{n+1}=r\frac{ra_{n-1}+a_n}{1+ra_{n-1}a_n},\quad n\ge 3.
\]
The following statements hold:
\begin{enumerate}
\item[(i)] $a_n\in(0,r)$ for all $n\ge 2$.

\item[(ii)] The sequence $\{a_n\}$ is monotone decreasing.

\item[(iii)] If $r\in(0,\frac{-1+\sqrt{5}}{2}]$, $\lim a_n=0$, while if $r\in(\frac{-1+\sqrt{5}}{2},1]$,
\[
\lim a_n=\sqrt{r+1-\frac1r}.
\]

\item[(iv)] When $r<\frac{-1+\sqrt{5}}2$,
\[
\frac{a_n}{a_{n-1}}<\frac{a_{n+1}}{a_{n-1}}=r\frac{r+\frac{a_n}{a_{n-1}}}{1+ra_{n-1}a_n}<r(r+1)<1.
\]

\item[(v)] Let $H=\limsup \frac{a_n}{a_{n-1}}$. We have
\[
H\le r(r+H)\Leftrightarrow H\le \frac{r^2}{1-r}<r(r+1)<1.
\]
\end{enumerate}
\end{lem}

\begin{proof} \textit{of Theorem \ref{cerosPerTres}.} From (\ref{recurrencia}) and (\ref{recurrencia2}), we have
\begin{multline*}
\Phi_{n+1}(z)=\\= z\left(z\left(z+\overline{\Phi_n(0)}\Phi_{n+1}(0)\right)+\left(\Phi_{n+1}(0)+z\Phi_n(0)\right) \overline{\Phi_{n-1}(0)}\right)\Phi_{n-2}(z) +\\ +\left(z\left(z+\overline{\Phi_n(0)}\Phi_{n+1}(0)\right)\Phi_{n-1}(0)+\Phi_{n+1}(0) +z\Phi_n(0)\right)\Phi_{n-2}^*(z).
\end{multline*}
If $\Phi_{n+1}$ and $\Phi_{n-2}$ have a common zero $\zeta$,
\[
\Phi_{n+1}(0)=-\zeta\frac{\zeta\Phi_{n-1}(0)+\Phi_n(0)}{1+\zeta \Phi_{n-1}(0)\overline{\Phi_n(0)}}.
\]

Hence,
\[
|\Phi_1(0)|=|\alpha|,\quad |\Phi_2(0)|\le|\beta|\frac{|\beta|+|\alpha|}{1+|\alpha||\beta|}\le \frac{2r^2}{1+r^2}, \quad |\Phi_3(0)|\le \frac{2r^2}{1+r^2},
\]
\[
\left|\Phi_{n+1}(0)\right|=|\zeta|\left|\frac{\zeta\Phi_{n-1}(0)+\Phi_n(0)}{1+\zeta \Phi_{n-1}(0)\overline{\Phi_n(0)}}\right|\le  r\frac{r|\Phi_{n-1}(0)|+|\Phi_n(0)|}{1+r |\Phi_{n-1}(0)||\Phi_n(0)|}.
\]

Let $\{a_n\}$ be as in the lemma above. Thus,
\[
\left|\Phi_{n+1}(0)\right|\le|a_n|,\quad n\ge 2.
\]
If $r\in(0,\frac{-1+\sqrt{5}}{2}]$, according to lemma above, $\lim a_n=0$. Hence, $\lim\Phi_n(0)=0$ holds.

If $r\in(0,\frac{-1+\sqrt{5}}{2})$,
\[
\limsup|\Phi_n(0)|^{1/n}\le \limsup|a_n(0)|^{1/n}\le \limsup\left|\frac{a_n}{a_{n-1}}\right|<\frac{r^2}{1-r}<1.
\]
\end{proof}

\section{Zeros' distance from the unit circle}\label{SecDisZerUC}

For the proof of Theorem \ref{TeoVelCirOPUC}, we require some auxiliary results.

\begin{lem} Let $\Lambda$ denote an infinite subset of the natural numbers. Let
\[
\{V_n(z)=\prod_{j=1}^{n}(z-v_{n,j}):n\in\Lambda\}
\]
be a sequence of monic polynomials whose zeros, $\{v_{n,j}\}$, all lie in $\mathbb{D}$ and such that
\[
\lim_{n\to\infty,\, n\in\Lambda}\frac{V_n(z)}{V_n^*(z)}=0,
\]
uniformly on compact subset  of $\mathbb{D}$. Suppose there are $z_0\in\mathbb{D}$ and $r>0$ such that
$V_n(z)\ne 0$ for all $z:|z-z_0|>r$. Then
\[
\lim_{n\to\infty,\, n\in\Lambda}\sum_{j=1}^n(1-|v_{n,j}|)=\infty.
\]
\end{lem}

\begin{proof} Without loss, we can assume $z_0=0$. In fact, by the simple change of variables
\[
z=\frac{\zeta+z_0}{1+\overline{z_0}\zeta}
\]
we obtain
\[
\frac{W_n(\zeta)}{W_n^*(\zeta)}=\frac{V_n(\frac{\zeta+z_0}{1+\overline{z_0}\zeta})} {V_n^*(\frac{\zeta+z_0}{1+\overline{z_0}\zeta})},
\]
where $W_n$ is a monic polynomials whose zeros, $\zeta_{n,j},\, j=1,\ldots,n$, lie in $\mathbb{D}$ and there exists $\delta>0$ such that
\[
|\zeta_{n,j}|>\delta\quad j=1,\ldots,n.
\]
Moreover, we have
\[
\lim_{n\to\infty,\, n\in\Lambda}\sum_{j=1}^n(1-|v_{n,j}|)=\infty\Leftrightarrow \lim_{n\to\infty,\, n\in\Lambda}\sum_{j=1}^n(1-|\zeta_{n,j}|)
\]
according to there are $k_1 = k_1(z_0)>0$, $k_2 = k_2(z_0)>0$ such that
\[
k_1(1-|\zeta|)\leq 1 -
\left|\frac{\zeta+z_0}{1+\overline{z_0}\zeta}\right|\leq
k_2(1-|\zeta|),\quad\forall\zeta\in\mathbb{D}.
\]
Thus, we assume $z_0=0$. By hypothesis,
\begin{equation}\label{EquationAux5}
\lim_{n\to\infty,\, n\in\Lambda}\prod_{j=1}^nv_{n,j}=\lim_{n\to\infty,\, n\in\Lambda}\frac{V_n(0)}{V_n^*(0)}=0\Leftrightarrow \lim_{n\to\infty,\, n\in\Lambda}\sum_{j=1}^n\log|v_{n,j}|=-\infty.
\end{equation}
Since $|v_{n,j}|\ge r$ and  there is $\alpha<-1$ such that $\alpha x<\log(1-x),\,\forall x\in(0,1-r),$ we have
\[
\alpha (1-|v_{n,j}|)<\log(1-(1-|v_{n,j}|))=\log|v_{n,j}|
\]
and the proof of the lemma follows from (\ref{EquationAux5}).
\end{proof}

\begin{rmk} A sequence of monic polynomials $\{V_n(z)=\prod(z-v_{n,j}),n=1,2,\ldots\}$, which zeros lie in $\mathbb{D}$ satisfying
\begin{equation}\label{EquationAux10}
\lim_{n\to\infty}\frac{V_n(z)}{V_n^*(z)}=0
\end{equation}
uniformly on compact subset of $\mathbb{D}$, plays a key role in rational approximation. In fact, condition (\ref{EquationAux10}) is equivalent to the set of rational functions
\[
\{\frac{p_n}{V_n^*}:p_n\textnormal{ polynomial of degree}\le n, n=1,2,\ldots\}
\]
is dense in the space of analytic function in $\overline{\mathbb{D}}$ with the uniform norm. See Corollary 2, p. 246, in \cite{Wal}.

\end{rmk}

The next result is a generalization of a Walsh's Theorem (see Theorem 9, p. 247, in \cite{Wal}).

\begin{lem}\label{LemaCovCeroSumaCeros} Let $\Lambda$ denote an infinite subset of the natural numbers and let
\[
\{V_n(z)=\prod_{j=1}^{n}(z-z_{n,j}):n\in\Lambda\}
\]
be a sequence of monic polynomials whose zeros, $\{v_{n,j}\}$, lie in $\mathbb{D}$. The following statements are equivalent:
\begin{equation}\label{EquationAux6}
\lim_{n\to\infty,\, n\in\Lambda}\frac{V_n(z)}{V_n^*(z)}=0,
\end{equation}
uniformly on compact subset of $\mathbb{D}$.
\begin{equation}\label{EquationAux7}
\lim_{n\to\infty,\, n\in\Lambda}\sum_{j=1}^n(1-|v_{n,j}|)=\infty.
\end{equation}
\end{lem}

\begin{proof} Assume (\ref{EquationAux7}) holds. Let $T\in(0,1)$ be fixed. We have
\[
\frac{1-T}{T+1}(1-|v_{n,j}|)\le \frac{(1-T)(1-|v_{n,j}|)}{1+T|v_{n,j}|}\le \frac{1-T}{T}(1-|v_{n,j}|).
\]
Thus, (\ref{EquationAux7}) is equivalent to
\[
\lim_{n\to\infty,\, n\in\Lambda}\sum_{j=1}^n\frac{(1-T)(1-|v_{n,j}|)}{1+T|v_{n,j}|}=\infty
\]
for each $T\in(0,1)$. As $\frac{(1-T)(1-|v_{n,j}|)}{1+T|v_{n,j}|}<1-T<1$, there exists $\lambda<-1$ such that
\begin{multline*}
\lambda\left(\frac{(1-T)(1-|v_{n,j}|)}{1+T|v_{n,j}|}\right) \le\log\left(1-\frac{(1-T)(1-|v_{n,j}|)}{1+T|v_{n,j}|}\right)\\ \le -\left(\frac{(1-T)(1-|v_{n,j}|)}{1+T|v_{n,j}|}\right)
\end{multline*}
\[
\Leftrightarrow \lambda\left(\frac{(1-T)(1-|v_{n,j}|)}{1+T|v_{n,j}|}\right) \le\log\left(\frac{T+|v_{n,j}|)}{1+T|v_{n,j}|}\right)\le -\left(\frac{(1-T)(1-|v_{n,j}|)}{1+T|v_{n,j}|}\right).
\]
Hence, (\ref{EquationAux7}) is equivalent to
\[
\lim_{n\to\infty,\, n\in\Lambda}\sum_{j=1}^n\log\left|\frac{T+|v_{n,j}|)}{1+T|v_{n,j}|}\right|=-\infty.
\]
If $|z|\le T$, we have
\[
\left|\frac{V_n(z)}{V_n^*(z)}\right|\le\prod_{j=1}^n\frac{T+|v_{n,j}|}{1+T|v_{n,j}|}.
\]
Therefore, if (\ref{EquationAux7}) holds, then we have (\ref{EquationAux6}).

According to Lemma above, to prove (\ref{EquationAux6}) implies (\ref{EquationAux7}), we need only show the following statement: Assume (\ref{EquationAux6}) holds and that for all infinite set $\Lambda_1\subset\Lambda$, all $z_0\in\mathbb{D}$, and for any $\epsilon>0$ there exists an infinity set $\Lambda_2\subset\Lambda_1$ such that for any $n\in\Lambda_2$ there is $j\in\{1,\ldots,n\}$ such that $|v_{n,j}|<\epsilon$, i.e., $V_{n}$ has a zero in $\{z:|z-z_0|<\epsilon\}$, then
\[
\lim_{n\to\infty,\, n\in\Lambda}\sum_{j=1}^{n}(1-|v_{n,j}|)=\infty.
\]

To get a contradiction, we assume that there exist $M>0$ and an infinite set $\Gamma\subset\Lambda$ such that
\begin{equation}\label{contradSuma}
\sum_{j=1}^{n}(1-|v_{n,j}|)\le M,\quad \forall n\in\Gamma.
\end{equation}
We can choose $w_1,\ldots,w_k$ in the circle $\{z:|z|=1/2\}$ and $r>0$ sufficiently small such that the disks  $\{z:|z-w_j|<r\},\, j=1,\ldots,k$ are disjoint and
\[
k(1/2-r)>M.
\]
According to our assumptions, we can choose an infinity set $\Gamma_1\subset \Gamma$ such that
\[
V_{n}(z)\textnormal{ has a zero in $\{z:|z-w_1|<r\}$ for all $n\in\Gamma_1$}.
\]

Given $\Gamma_1$, we can choose $\Gamma_2\subset\Gamma_1\subset\Gamma$ such that
\[
V_{n}(z)\textnormal{ has a zero in $\{z:|z-w_2|<r\}$ for all $n\in\Gamma_2$},
\]
so, $V_n,\, n\in\Gamma_2$, has a zero in $\{z:|z-w_2|<r\}$ and in $\{z:|z-w_1|<r\}$. In this way, there exists an infinity set $\Gamma_k\subset \Gamma$ such that
\[
V_{n}(z)\textnormal{ has a zero in $|z-w_j|<r$ for all $n\in\Gamma_k$ and $j=1,\ldots,k$.}
\]
But,  just as we have been chosen $w_1,\ldots,w_k$ and $r$, we get a contradiction. In fact, because $\Gamma_k\subset\Gamma$ and for $n\in\Gamma_k$
\[
M\ge\sum_{j=1}^n(1-|v_{n,j}|)\ge \sum_{j:|v_{n,j}-w_l|<r,\,l=1,\ldots,k}(1-|v_{n,j}|)>k(1/2-r)>M.
\]
\end{proof}

\begin{proof}\textit{of Theorem \ref{TeoVelCirOPUC}.} It is very well known that $\lim\Phi_n(0)=0$ is equivalent to
\[
\lim_{n\to\infty}\frac{\Phi_n(z)}{\Phi_n^*(z)}=0,
\]
uniformly on compact subset of $\mathbb{D}$, see, for example, Theorem 1.7.4, p. 91 in \cite{SimonTI}. Therefore, Theorem \ref{TeoVelCirOPUC} follows immediately form Lemma \ref{LemaCovCeroSumaCeros}.

\end{proof}

\subsection{Zero's distance on an arc of the unit circle}

The paper \cite{MarMclSaf2} proved
\[
|z_{n,j}|=1-\frac{\log n}{n}+O(\frac1{n}).
\]
for the zeros of OPUCs whose orthogonality measure is $d\mu(\theta)=W(e^{i\theta})d\theta$, with $W(z) = w(z)\prod_{k=1}^m
|z - a_k|^{2\beta_k}$,\, $|z| = 1,\, |a_k| = 1, \, \beta_k> -1/2,\, k = 1,\ldots,m$, where $w(z) > 0$,  $|z| = 1$, has analytic continuation to an annulus around the unit circle.

It is known that for positive weight almost everywhere on an arc, $\Delta$, of the unit circle, the zeros of their orthogonal polynomials approach to the unit circle (see \cite{Bel-Lop}) as the degree of the polynomials increasing. Moreover, see \cite{Bel-Min}, on each neighborhood of each arc, $\Delta'\subset\Delta$, their exist $O(n)$ zeros of $\Phi_n$ for $n$ large enough.

Next, we find the rate  with which the zeros of OPUCs of the Chebyshev weight on an arc of the unit circle approach to $\partial(\mathbb{D})$. Consider the weight
\[
w(\theta)=\frac{\sin(\alpha/2)}{2\sin(\theta/2)\sqrt{\cos^2\alpha/2-\cos^2\theta/2}},\quad \theta\in[\alpha,2\pi-\alpha],
\]
let $\{\Phi_n\}$ be the sequence of monic orthogonal polynomial for $w$ and $\{z_{n,j}:j=1,\ldots,n\}$ are their zeros.

\begin{teo} If $\lim_{n\to\infty}z_{n,j_n}=e^{i\theta_0}$ with $\theta_0\in[\alpha,2\pi-\alpha]$,
\[
|z_{n,j_n}|=1-f(\theta_0)/n+O(1/n^2),
\]
where $f$ is a positive continuous function in $[\alpha,2\pi-\alpha]$ which is nonzero in $(\alpha,2\pi-\alpha)$.
\end{teo}

\begin{proof} We have
\begin{equation}\label{OPUCarco}
\varphi_n(z)=K_n\,\left\{\frac{w^n(v)}{1-\beta v}+\frac{v\, w^n(1/v)}{v-\beta}\right\},\quad z=h(v),
\end{equation}
where $\beta=i\tan \frac{\pi-\alpha}{4}$,
\[
w(v)=i\frac{1-\beta v}{v+\beta},\quad w(1/v)=i\frac{v-\beta }{1+\beta v},\quad
z=h(u)=\frac{(v-\beta)(\beta v-1)}{(v+\beta)(\beta v+1)}.
\]
See \cite{Gol}. Also, these polynomials were studied by Akhiezer.

The function $w=w(v)$ is an invertible analytic homeomorphism of $\mathbb{D}$ to $\mathbb{D}$
and $z=h(u)=\frac{(v-\beta)(\beta v-1)}{(v+\beta)(\beta v+1)}$ is analytic in $\mathbb{C}\setminus\{-\beta,-\frac1\beta\}$ and a homeomorphism of $\mathbb{D}\setminus\{-\beta\}$ to $\mathbb{C}\setminus\Delta_\alpha$ where, $\Delta_\alpha=\{e^{i\theta}:\theta\in[\alpha,2\pi-\alpha\}$.

Each zero $z_{n,j}$ of (\ref{OPUCarco}) corresponds with an unique $v_{n,j}: h(v_{n,j})=z_{n,j},$ $|v_{n,j}|<1$. Since $z_{n,j}\to \Delta_\alpha$, as $n\to\infty$, we have $|v_{n,j}|\to 1$. Moreover, we know
\begin{equation}\label{EquationAux8}
\frac{w^n(v_{n,j})}{1-\beta v_{n,j}}+\frac{v_{n,j}\, w^n(1/v_{n,j})}{v_{j,n}-\beta}=0
\Leftrightarrow
\frac{w^n(1/v_{n,j})}{w^n(v_{n,j})}=-\frac{v_{n,j}-\beta}{v_{n,j}(1-\beta v_{n,j})}.
\end{equation}

Consider $e^{i\theta_0}\in\Delta_{\alpha}$ with $\theta_0\in(\alpha,2\pi-\alpha)$ and $e^{i\theta_0}=\lim_{n}z_{n,j}$, here $j=j_n$ changes with $n$. Actually, we have $n\in\Lambda$ a sequence of indexes such that $z_{n,j}, n\in\Lambda,$ has limit $e^{i\theta_0}$. Throughout, we consider such indexes. Thus,
\[
\lim_{n}v_{n,j}=e^{i\omega_0},\quad e^{i\theta_0}=h(e^{i\omega_0}),\quad \omega_0\in(0,\pi),
\]
\[
\lim_n \Im(v_{n,j})=\sin\omega_0,
\]
\[
\lim_n \frac1{|v_{n,j}|^2} \frac{|v_{n,j}|^2-2\Re(\overline{v_{n,j}}\beta)+|\beta|^2}{1-2\Re(v_{n,j}\beta)+|v_{n,j}\beta|^2} = \frac{1-2\tan\eta\sin\omega_0+|\beta|^2}{1+2\tan\eta\sin\omega_0+|\beta|^2}\in(0,1).
\]

On the other hand,
\begin{multline*}
\left|\frac{v_{n,j}-\beta}{v_{n,j}(1-\beta v_{n,j})}\right|^2=\frac1{|v_{n,j}|^2}\frac{(v_{n,j}-\beta) (\overline{v_{n,j}}-\overline{\beta})}{(1-\beta v_{n,j})(1-\overline{\beta v_{n,j}})}\\= \frac1{|v_{n,j}|^2} \frac{|v_{n,j}|^2-2\Re(\overline{v_{n,j}}\beta)+|\beta|^2}{1-2\Re(v_{n,j}\beta)+|v_{n,j}\beta|^2},
\end{multline*}
and $\Re(v_{n,j}\beta)=-\Im(v_{n,j})\tan\eta$, $\Re(\overline{v_{n,j}}\beta)=\Im(v_{n,j})\tan\eta$.

Moreover,
\[
|w(1/v_{n,j})|^2=\left|\frac{1-\beta v_{n,j}}{v_{n,j}+\beta }\right|^2=\frac{1-2\Re(v_{n,j}\beta) +|v_{n,j}\beta|^2}{|v_{n,j}|^2+2\Re(\overline{v_{n,j}}\beta)+|\beta|^2}\to 1,\quad \text{as }n\to\infty,
\]
\begin{multline*}
|w(1/v_{n,j})|^2-1\\=\frac{(1-|v_{n,j}|^2)(1-|\beta|^2)} {1+2\Re(\overline{v_{n,j}}\beta)+|\beta|^2}\sim 2 (1-
|v_{n,j}|)\frac{(1-|\beta|^2)} {1+2 \tan\eta\sin \omega_0+\tan^2\eta}.
\end{multline*}
\[
|w(v_{n,j})|^2=\left|\frac{v_{n,j}-\beta }{1+\beta v_{n,j}}\right|^2=\frac{|v_{n,j}|^2-2\Re(\overline{v_{n,j}}\beta) +|\beta|^2}{1+2\Re(v_{n,j}\beta)+|v_{n,j}\beta|^2}\to 1,\quad \text{as }n\to\infty,
\]
\[
|w(v_{n,j})|^2-1=\frac{(|v_{n,j}|^2-1)(1-|\beta|^2)} {1+2\Re(v_{n,j}\beta)+|v_{n,j}\beta|^2}\sim -2 (1-|v_{n,j}|)\frac{(1-|\beta|^2)} {1-2 \tan \eta\sin \omega_0+|\beta|^2}.
\]

Hence,
\begin{multline*}
\left|\frac{w(1/v_{n,j})}{w(v_{n,j})}\right|^2-1=\frac{|w(1/v_{n,j})|^2-|w(v_{n,j})|^2}{|w(v_{n,j})|^2} =\frac{-2(1-|\beta|^2)(1-|v_{n,j}|)}{|w(v_{n,j})|^2}\\ \underset{n\to\infty}{\rightarrow}\frac{4\tan \eta\sin \omega_0}{(1+2 \tan \eta\sin \omega_0+|\beta|^2)(1-2 \tan \eta\sin \omega_0+|\beta|^2)}\\ =\frac{-8(1-|\beta|^2)(1-|v_{n,j}|) \tan \eta\sin \omega_0}{(1-2 \tan \eta\sin \omega_0+|\beta|^2)(1+2 \tan \eta\sin \omega_0+|\beta|^2)}.
\end{multline*}

From (\ref{EquationAux8}), we obtain
\[
\left|\frac{w(1/v_{n,j})}{w(v_{n,j})}\right|^n\underset{n}{\to} \frac{1-2\tan\eta\sin\omega_0+|\beta|^2}{1+2\tan\eta\sin\omega_0+|\beta|^2},
\]
so,
\[ n\left(\left|\frac{w(1/v_{n,j})}{w(v_{n,j})}\right|-1\right)\underset{n}{\rightarrow } \log\left(\frac{1-2\tan\eta\sin\omega_0+|\beta|^2}{1+2\tan\eta\sin\omega_0+|\beta|^2}\right),
\]
Thus,
\[
\lim_nn\, (1-|v_{n,j}|)= f(\omega_0)
\]
where
\begin{multline*}
f(\omega_0)=\frac{(1-2 \tan \eta\sin \omega_0+|\beta|^2)(1+2 \tan \eta\sin \omega_0+|\beta|^2)}{8(1-|\beta|^2)\tan \eta\sin \omega_0} \\ \times \log\left(\frac{1+2\tan\eta\sin\omega_0+|\beta|^2}{1-2\tan\eta\sin\omega_0+|\beta|^2}\right).
\end{multline*}

Therefore,
\begin{multline*}
|z_{n,j}|=|h(v_{n,j})|=|w(v_{n,j})||w(\frac1{v_{n,j}})|\\=(1+|w(v_{n,j})|-1)(1+|w(\frac1{v_{n,j}})|-1) \\ \sim\left(1+\frac{2 (1-|v_{n,j}|)(1-|\beta|^2)} {1+2 \tan \eta\sin \omega_0+|\beta|^2}\right)\left(1+\frac{2 (|v_{n,j}|-1)(1-|\beta|^2)} {1-2 \tan\eta\sin \omega_0+\tan^2\eta}\right)\\ \sim 1-\frac1n \widetilde{f}(\omega_0),
\end{multline*}
where
\[
\widetilde{f}(\omega_0)= \log\left(\frac{1+2\tan\eta\sin\omega_0+\tan^2\eta}{1-2\tan\eta\sin\omega_0+\tan^2\eta}\right).
\]
\end{proof}

\begin{rmk}
The Figures \ref{GrafCerosPerDos}--\ref{GrafCerosPerTres2} were generated in Mathematica 6.
\end{rmk}

\begin{figure}
\begin{center}
\includegraphics[scale=0.7]{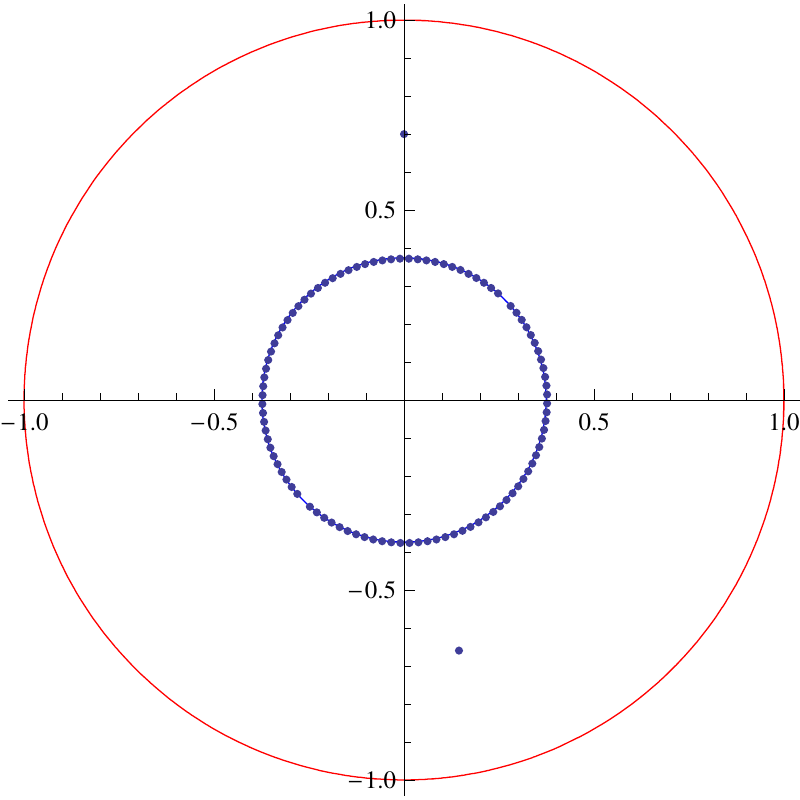}
\end{center}
\caption{\label{GrafCerosPerDos} Zeros of $\Phi_{100}$ for two period zeros: $0.2$ and $0.7i$.}
\end{figure}
\begin{figure}
\begin{center}
\includegraphics[scale=0.7]{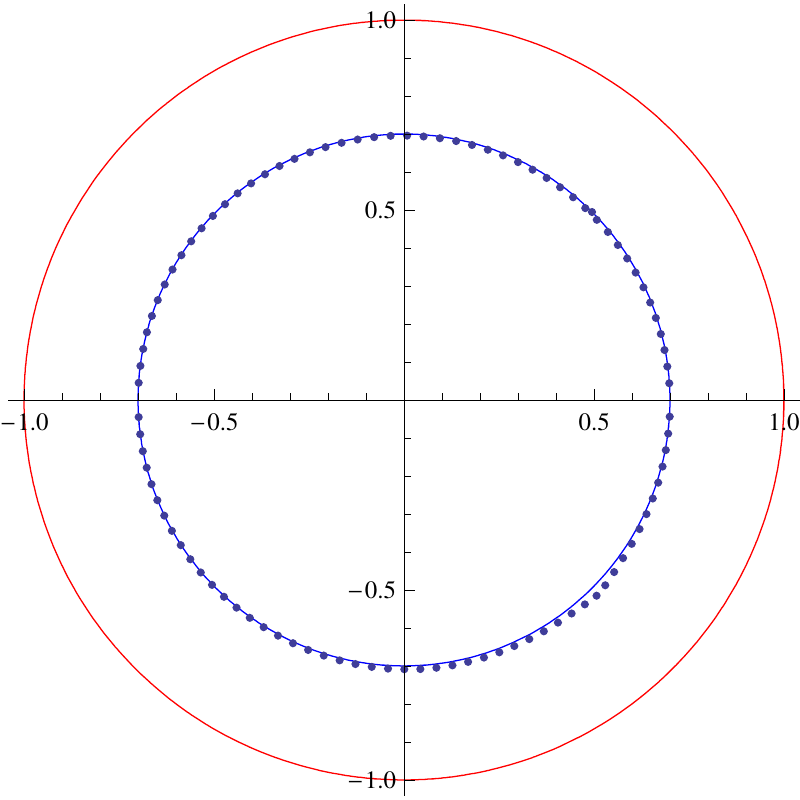}
\end{center}
\caption{\label{GrafCerosPerDos2} Zeros of $\Phi_{100}$ for two period zeros: $0.7e^{-i\frac\pi4}$ and $0.7e^{i\frac\pi4}$.}
\end{figure}


\begin{figure}
\begin{center}
\includegraphics[scale=0.7]{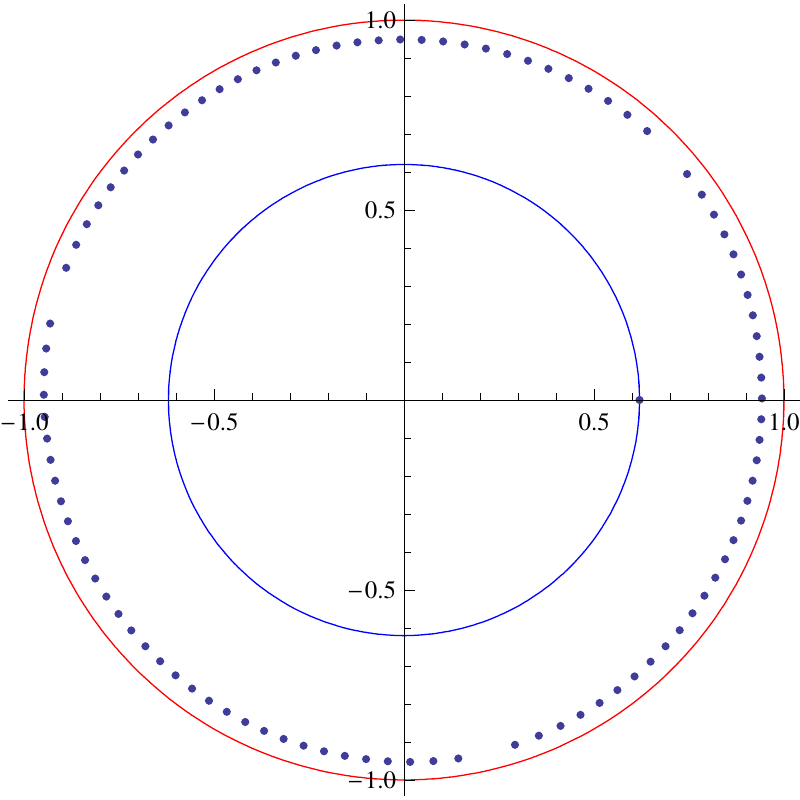}
\end{center}
\caption{\label{GrafCerosPerTres} Zeros of $\Phi_{100}$ for three period zeros:  $0.62$, $0.62e^{i\frac{2\pi}3}$ and $0.62 e^{-i\frac{2\pi}3}$. Observe $\frac{-1+\sqrt{5}}2=0.618034\ldots<0.62$}
\end{figure}

\begin{figure}
\begin{center}
\centering\includegraphics[scale=0.7]{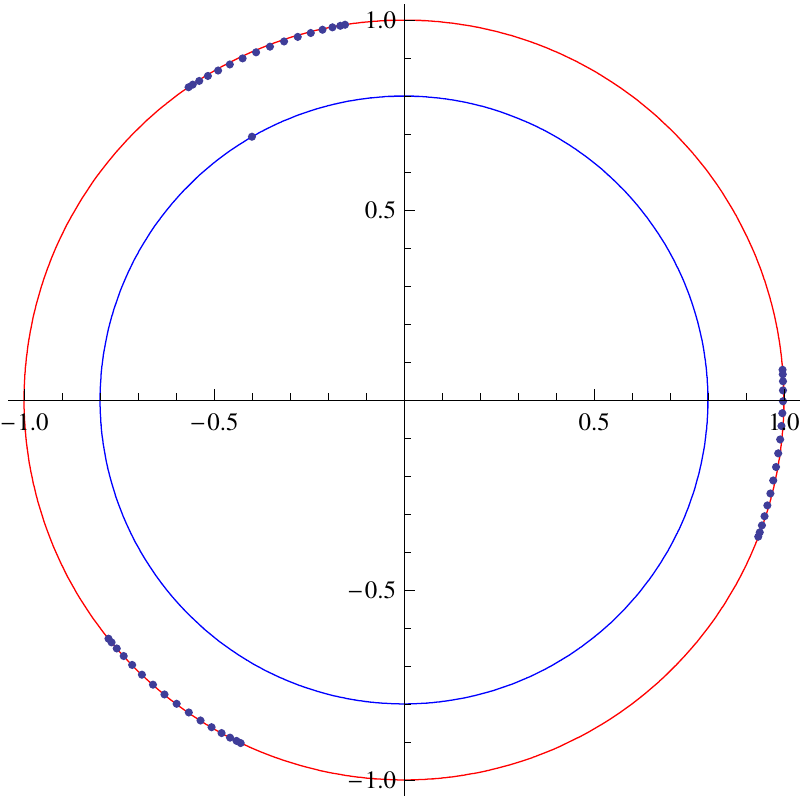}
\end{center}
\caption{\label{GrafCerosPerTres2} Zeros of $\Phi_{50}$ for three period zeros: $0.8$, $0.8e^{i\frac{2\pi}3}$ and $0.8 e^{-i\frac{2\pi}3}$. When the degree of the OPUC is larger than 50  calculating the zeros appear numerical instability in Mathematica}
\end{figure}


\end{document}